\documentclass[a4paper,11pt]{article}
\usepackage[latin1]{inputenc}
\usepackage[T1]{fontenc}

\usepackage{array}
\usepackage{fancyhdr}
\usepackage{longtable}
\usepackage{tabularx}
\usepackage{rotating}

\usepackage{makeidx}
\usepackage{amsmath}
\usepackage{amsthm}
\usepackage{amssymb}
\usepackage[mathscr]{eucal}
\usepackage{enumerate}
\usepackage[dvips]{epsfig}
\usepackage{float}  
\usepackage{ifthen} 
\usepackage{extarrows}  
\usepackage{stmaryrd}   
\usepackage{arydshln}     
\usepackage{multirow}
\usepackage{mathtools}
\usepackage[dvips]{geometry} 
\usepackage{listings}
\usepackage{subfig}
\usepackage{dsfont}
\usepackage{xspace}

\theoremstyle{plain}
\newtheorem{thm}{Theorem}[section]
\newtheorem{prop}[thm]{Proposition}

\theoremstyle{definition}

\theoremstyle{remark}

\newcommand{\Prob}{\mathrm{P}}
\newcommand{\Erw}{\mathrm{E}}

	\usepackage[bookmarksopen,bookmarksnumbered,colorlinks=true,linkcolor=red,citecolor=blue]{hyperref}
\usepackage{bookmark}
\hypersetup{
	pdfauthor={Jan Pleis, Andreas R\"o{\ss}ler},
	% ...
}

\title{{A Note on Some Martingale Inequalities}}

\author{Jan Pleis\thanks{The author was supported by 
		the Graduate School for Computing in Medicine and Life Sciences funded by
		Germany's Excellence Initiative [DFG GSC 235/2]. e-mail: pleis@math.uni-luebeck.de}
\ \ and Andreas R\"o{\ss}ler\thanks{e-mail: roessler@math.uni-luebeck.de}
\bigskip
\\
\small{
	Institute of Mathematics, Universit\"at zu L\"ubeck,} \\
\small{Ratzeburger Allee 160, 23562 L\"ubeck, Germany} 
}

\date{}

\begin{document}

\maketitle

\begin{abstract}
\noindent
We derive inequalities for time-discrete and time-continuous martingales 
that are similar to the well-known Burkholder inequalities. For the time-discrete 
case arbitrary martingales in $L^p(\Omega)$ are treated, whereas in the
time-continuous case martingales defined by It{\^o} integrals w.r.t.\ 
a multi-dimensional Wiener process are considered. The estimates for the
time-discrete martingales are related to the more general results by I.~Pinelis (1994) and are proved to be sharp by a different and more 
elementary proof for this special setting.
Further, for time-continuous martingales the presented inequalities are
generalizations of similar estimates 
proved by M.~Zakai (1967) and E.~Rio (2009) to the general multi-dimensional case.
Especially, these inequalities possess smaller constants compared to the ones 
that result if the original Burkholder inequalities would be applied for such 
estimates. Therefore, the presented inequalities are highly 
valuable in, e.g., stochastic analysis and stochastic numerics. 
\end{abstract}
%
%
%
%
\section{Burkholder inequalities for martingales}
In stochastic analysis the Burkholder inequalities give powerful estimates 
for martingales, and they are used frequently. 
In the following, let $(\Omega, \mathcal{F}, \Prob)$ be a complete probability space.
Let $p \in {[2,\infty[}$, $d, N \in \mathbb{N}$, 
and let $(\mathcal{G}_n)_{n \in \{0,1, \ldots,N\}}$
be a filtration. 
Consider a time-discrete martingale $(M_n)_{n\in\{0,1,\dots,N\}}$ 
in $L^{p}(\Omega;\mathbb{R}^d)$ w.r.t.\ the filtration
$(\mathcal{G}_{n})_{n\in\{0,1,\dots,N\}}$.
Define $d_0 = M_0$ and $d_k=M_k-M_{k-1}$ for $k \in \{1,\ldots,N\}$ such
that $M_n = \sum_{k=0}^{n} d_k$. According to \cite{Bur1988a}, the Burkholder
inequalities state
\begin{align} \label{eq:dBurkholder1}
	\| M_n \|_{L^p(\Omega; \mathbb{R}^d)}^2 \le (p-1)^2 
	\bigg\| \sum_{k=0}^{n} \| d_k \|^2 \bigg\|_{L^{\frac{p}{2}}(\Omega; \mathbb{R})}
\end{align}
and
\begin{align} \label{eq:dBurkholder2}
	\bigg\| \sup_{\nu \in\{0,1,\dots,n\}} \| M_{\nu} \| 
	\bigg\|_{L^p(\Omega; \mathbb{R})}^2
	\le p^2 \bigg\| \sum_{k=0}^{n} \| d_k \|^2 \bigg\|_{L^{\frac{p}{2}}(\Omega;
	\mathbb{R})}
\end{align}
for all $n \in \{0,1,\dots,N\}$ where the constants are best possible. 

%
These inequalities carry over to time-continuous martingales defined by 
It{\^o} stochastic integrals~\cite{Bur1988a}. 
Let $(\mathcal{F}_t)_{t \in [0,T]}$ for some $T>0$ be a filtration that fulfills 
the usual conditions. 
For $m\in\mathbb{N}$, let $W \colon [0,T] \times \Omega \to \mathbb{R}^m$ 
be an $m$-dimensional Wiener process w.r.t.\ the filtration 
$(\mathcal{F}_t)_{t \in [0,T]}$.
For $j \in \{1,\dots,m\}$, let $f^j \colon [0,T] \times \Omega \to \mathbb{R}^d$ 
be an $(\mathcal{F}_t)_{t\in[0,T]}$-adapted and measurable stochastic process with
$\Erw \big[ \big( \int_0^T  \| f^j_u \|^2 \, \mathrm{d}u \big)^{\frac{p}{2}} \big] < \infty$.
Then, it holds
\begin{align} \label{eq:cBurkholder1}
	\bigg\| \sum_{j=1}^m \int_{0}^t f^{j}_u \, \mathrm{d} W^{j}_u \bigg\|_{L^{p}(\Omega;
	\mathbb{R}^d)}^2 &\le (p-1)^2 \bigg\| \int_{0}^t \sum_{j=1}^m \big\| f^{j}_u 
	\big\|^2 \, \mathrm{d} u \bigg\|_{L^{\frac{p}{2}}(\Omega; \mathbb{R})}
\end{align}
and
\begin{align} \label{eq:cBurkholder2}
	\bigg\| \sup_{s\in [0,t]} \bigg\| \sum_{j=1}^m \int_{0}^s f^{j}_u \, \mathrm{d} 
	W^{j}_u \bigg\| \bigg\|_{L^{p}(\Omega; \mathbb{R})}^2 
	&\le p^2 \bigg\| \int_{0}^t \sum_{j=1}^m \big\| f^{j}_u \big\|^2 \, \mathrm{d} u
	\bigg\|_{L^{\frac{p}{2}}(\Omega; \mathbb{R})}
\end{align}
for all $t\in[0,T]$, where the constants are best possible, see \cite{Bur1988a}. 

%
Considering for example the convergence analysis of numerical methods for
stochastic 
(partial) differential equations, the expressions on the right-hand 
sides of inequalities~\eqref{eq:dBurkholder1}-\eqref{eq:cBurkholder2} are 
usually not needed explicitly. 
This is because the triangle inequality is often applied to the 
$L^{\frac{p}{2}}(\Omega; \mathbb{R})$-norms in order to obtain suitable 
upper bounds.
In the following, we show that for such inequalities the constants can be 
reduced in case of $p>2$ compared to the constants given in 
\eqref{eq:dBurkholder1}-\eqref{eq:cBurkholder2}.
\section{Generalized Burkholder-type inequalities}
In this section, we generalize some Burkholder-type inequalities proved 
by E.~Rio~\cite{Rio2009a} and M.~Zakai~\cite{Zak1967a}, respectively, 
to the multi-dimensional case. The resulting constants for these inequalities are 
smaller than the ones that one would get from
inequalities~\eqref{eq:dBurkholder1}-\eqref{eq:cBurkholder2}. 
The proofs for these inequalities are postponed to Section~\ref{sec:proofs}. 
Firstly, we consider the time-discrete case. 
\begin{prop} \label{thm:discreteBurkholder}
	Let $p\in {[2,\infty[}$ and $N \in \mathbb{N}$. Further, let 
	$(M_n)_{n\in\{0,1,\dots,N\}}$ with 
	$M_n = \sum_{k=0}^{n} d_k$ be a time-discrete 
	martingale in $L^{p}(\Omega; \mathbb{R}^d)$ 
	w.r.t.\ the filtration $(\mathcal{G}_{n})_{n \in \{0,1,\dots,N\}}$. 
	Then, it holds 
	\begin{align} \label{eq:discreteBurkholder}
		\| M_n \|_{L^{p}(\Omega; \mathbb{R}^d)}^2 \le (p-1) \sum_{k=0}^{n} 
		\| d_k \|_{L^{p}(\Omega; \mathbb{R}^d)}^2
	\end{align}
	and
	\begin{align*}
		\bigg\| \sup_{\nu \in\{0,1,\dots,n\}} \| M_{\nu} \|
		\bigg\|_{L^{p}(\Omega; \mathbb{R})}^2 \le \frac{p^2}{p-1} \sum_{k=0}^{n} 
		\| d_k \|_{L^{p}(\Omega; \mathbb{R}^d)}^2
	\end{align*}
	for all $n \in \{0,1,\dots,N\}$. The constants are best possible.
\end{prop}
Here, it has to be pointed out that the result
\eqref{eq:discreteBurkholder} is a special case of \cite[Prop.~2.5]{MR1331198}
by Pinelis if $d=1$, and the proof given in \cite{MR1331198}
can be easily generalized to the case $d \geq 1$ that we consider. Nevertheless, 
our poof for Proposition~\ref{thm:discreteBurkholder} is rather 
elementary compared to the one given for the more general setting in
\cite{MR1331198}. Thus, the presented proof may be of some interest on its own.
In the case of $d=1$, an inequality similar to~\eqref{eq:discreteBurkholder} and its 
sharpness are proved in~\cite{MR1331198}. In Section~\ref{sec:proofs}, 
we extent this proof to general $d \in \mathbb{N}$.
Next, we consider the time-continuous case.
Already in 1967, Zakai proved inequality~\eqref{eq:Zakai} 
of the following proposition in the case of $d=m=1$, see
\cite[Theorem~1]{Zak1967a}. 
We generalize this idea to $\mathbb{R}^d$-valued integrands of It{\^o} 
stochastic integrals w.r.t.\ an $m$-dimensional Wiener process.
\begin{prop} \label{thm:Zakai}
	Let $p\in {[2,\infty[}$, and for $j \in \{1,\dots,m\}$, let 
	$f^j \colon [0,T] \times \Omega \rightarrow \mathbb{R}^d$ be an 
	$(\mathcal{F}_t)_{t\in[0,T]}$-adapted, measurable stochastic process 
	with
	$\Erw \big[ \big( \int_0^T \| f^j_u \|^2 \, \mathrm{d} u \big)^{\frac{p}{2}}
	\big] <	\infty$.
	Then, it holds
	\begin{align} \label{eq:Zakai}
		\bigg\| \sum_{j=1}^m \int_{0}^t f^{j}_u \, \mathrm{d} W^{j}_u 
		\bigg\|_{L^{p}(\Omega; \mathbb{R})}^2
		&\le (p-1) \int_{0}^t \bigg\| \sum_{j=1}^m \big\| f^{j}_u \big\|^2
		\bigg\|_{L^{\frac{p}{2}}(\Omega; \mathbb{R})} \, \mathrm{d} u
	\end{align}
	and
	\begin{align*}
		\bigg\| \sup_{s\in [0,t]} \bigg\| \sum_{j=1}^m \int_{0}^s f^{j}_u 
		\, \mathrm{d} W^{j}_u \bigg\| \bigg\|_{L^{p}(\Omega; \mathbb{R})}^2
		&\le \frac{p^2}{p-1} \int_{0}^t \bigg\| 
		\sum_{j=1}^m \big\| f^{j}_u \big\|^2 
		\bigg\|_{L^{\frac{p}{2}}(\Omega; \mathbb{R})} \, \mathrm{d} u
	\end{align*}
	for all $t\in[0,T]$.
\end{prop}
Here, it is an open problem whether the constants in Proposition~\ref{thm:Zakai} 
are best possible. Since we only have one-sided estimates, we cannot 
transfer the results from the time-discrete case in 
Proposition~\ref{thm:discreteBurkholder} to the time-continuous case 
as Burkholder did in \cite{Bur1988a}, cf.\ \cite{Dol1969a}.
However, for $p>2$ the constants in Proposition~\ref{thm:discreteBurkholder} 
and Proposition~\ref{thm:Zakai} are of order $\mathcal{O}(p)$ as 
$p \to \infty$, and they are smaller than the ones in 
inequalities~\eqref{eq:dBurkholder1}-\eqref{eq:cBurkholder2}, 
where the constants are of order $\mathcal{O}(p^2)$ as $p \to \infty$. 
This makes the inequalities in Proposition~\ref{thm:discreteBurkholder} and 
Proposition~\ref{thm:Zakai} highly valuable for, e.g., the convergence analysis of
numerical methods for stochastic (partial) differential equations where 
constants matter.
\section{Proofs} \label{sec:proofs}
\begin{proof}[Proof of Proposition~\ref{thm:discreteBurkholder}]
	In the case of $p=2$, the assertion follows from the discrete Burkholder 
	inequalities~\eqref{eq:dBurkholder1} and~\eqref{eq:dBurkholder2}.
	Therefore, we assume $p\in {]2,\infty[}$ in the following. 
	An inequality similar to the one in~\eqref{eq:discreteBurkholder} and 
	its sharpness was proved by E.~Rio in \cite[Section~2]{Rio2009a} for 
	the case of $d=1$. 
	
	We adapt his proof to the case of general $d\in\mathbb N$.
	For this, we generalize \cite[Proposition~2.1]{Rio2009a}. 
	Let $X, Y \in L^{p}(\Omega; \mathbb{R}^d)$ and let $\mathcal{G} \subseteq
	\mathcal{F}$ be some sub-$\sigma$-algebra such that $X$ is
	$\mathcal{G}$-measurable and 
	$\Erw [Y \vert \mathcal{G}]=0$ $\Prob$-a.s. Then, we first prove that
	\begin{align} \label{eq:Rio}
		\| X+Y \|_{L^{p}(\Omega; \mathbb{R}^d)}^2 
		\le \| X \|_{L^{p}(\Omega; \mathbb{R}^d)}^2
		+(p-1) \| Y \|_{L^{p}(\Omega; \mathbb{R}^d)}^2.
	\end{align}
	If $X=0$ or $Y=0$, this inequality is clearly true, so we assume 
	$\| X \|_{L^{p}(\Omega; \mathbb{R}^d)} > 0$ and $\| Y 
	\|_{L^{p}(\Omega; \mathbb{R}^d)} > 0$. Define the 
	function $\varphi \colon [0,1] \to \mathbb{R}$ by $\varphi(t)=\| x + ty \|^p$ 
	for $x,y \in \mathbb{R}^d$. Using Taylor expansion, it holds
	$\varphi(1) = \varphi(0) + \varphi'(0) + \int_0^1 \varphi''(t) (1-t) \, \mathrm{d} t$
	and thus
	\begin{align*}
		\| X+Y \|^p &= \| X \|^p + p \| X \|^{p-2} \sum_{i=1}^d X^i Y^i 
		+ p \int_0^1 \| X+tY \|^{p-2} \| Y \|^2 (1-t) \, \mathrm{d} t \\
	&
	\phantom{=\| X \|^p}~+ p(p-2) \int_0^1 \| X+tY \|^{p-4} \bigg( \sum_{i=1}^d
	(X^i+tY^i) Y^i \bigg)^2 (1-t) \, \mathrm{d} t
	\end{align*}
	$\Prob$-a.s. Considering the integrand of the last 
	integral in the Taylor expansion above, it holds
	\begin{align*}
		\| X+tY \|^{p-4} \bigg( \sum_{i=1}^d (X^i+tY^i) Y^i \bigg)^2
		\le \| X+tY \|^{p-2} \| Y \|^2
	\end{align*}
	$\Prob$-a.s.\ by Cauchy-Schwarz inequality. Then, it follows
	\begin{align} \label{eq:normXYtaylor1}
		\begin{split}
			\| X+Y \|^p &\le \| X \|^p + p \| X \|^{p-2} \sum_{i=1}^d X^i Y^i
			+ p(p-1)\! \int_0^1 \| X+tY \|^{p-2} \| Y \|^2 (1-t) \, \mathrm{d} t
	\end{split}
	\end{align}
	$\Prob$-a.s. Due to the assumptions, it holds
	\begin{align*}
		\Erw \big[ \Erw[ \| X \|^{p-2} X^i Y^i \vert \mathcal{G} ] \big]
		= \Erw \big[ \| X \|^{p-2} X^i \Erw [ Y^i \vert \mathcal{G} ] \big]
		= 0
	\end{align*}
	and because
	\begin{align*}
		\Erw \big[ \| X+tY \|^{p-2} \| Y \|^2 \big] 
		\le \big( \Erw [ \| X+tY \|^p ] \big)^{\frac{p-2}{p}} \big( \Erw [ \| Y \|^p ]
		\big)^{\frac{2}{p}}
	\end{align*}
	by H\"older's inequality with $\frac{p-2}{p}+\frac{2}{p}=1$,
	taking the expectation on both sides of 
	inequality~\eqref{eq:normXYtaylor1} and applying Fubini's theorem,
	we obtain
	\begin{align*}
		&\Erw [ \| X+Y \|^p ] \le \Erw [ \| X \|^p ] + p(p-1) \int_0^1 \big(
		\Erw [ \| X+tY \|^p ] \big)^{\frac{p-2}{p}} \big( \Erw [ \| Y \|^p ] 
		\big)^{\frac{2}{p}} (1-t) \, \mathrm{d} t .
	\end{align*}
	This is a multi-dimensional version of \cite[Inequality~(2.1)]{Rio2009a}. 
	Now, we use a Gronwall-type inequality, i.e., we apply 
	\cite[Lemma on p.~171]{Zak1967a} with $\alpha=\frac{2}{p}$. 
	Then, it follows
	\begin{align*}
		&\Erw [ \| X+Y \|^p ] \le \bigg( \Erw [ \| X \|^p ]^{\frac{2}{p}}
		+ \frac{2}{p}p(p-1) \int_0^1 (1-t) \, \mathrm{d} t \,
		\big( \Erw [ \| Y \|^p ] \big)^{\frac{2}{p}} \bigg)^{\frac{p}{2}},
	\end{align*}
	and, since $\int_0^1(1-t) \, \mathrm{d} t=\frac{1}{2}$, inequality~\eqref{eq:Rio} 
	holds by raising both sides of the inequality above to the power of $\frac{2}{p}$. 
	Due to \cite[Remark~2.1]{Rio2009a}, the constant $p-1$ in 
	inequality~\eqref{eq:Rio} is best possible. 
	
	We remark that the considerations after \cite[inequality~(2.1)]{Rio2009a} 
	on \cite[p.~150]{Rio2009a} prove essentially Zakai's Gronwall-type 
	inequality in \cite[Lemma on p.~171]{Zak1967a}. 
	
	Now, we consider inequality~\eqref{eq:discreteBurkholder}. Since 
	$(M_n)_{n\in\{0,1,\dots,N\}}$ is a martingale, it holds that
	$\Erw [M_{n} \vert \mathcal{G}_{n-1} ] = M_{n-1}$ $\Prob$-a.s.\ for 
	$n \in \{1,\dots,N\}$, i.e., $\Erw [d_{n} \vert \mathcal{G}_{n-1} ] = 0$ $\Prob$-a.s. 
	Thus, inequality~\eqref{eq:discreteBurkholder} 
	follows from applying inequality~\eqref{eq:Rio} to $M_n= M_{n-1}+d_{n}$ 
	by induction on $n \in \{1,\dots,N\}$, cf.\ \cite[Theorem~2.1]{Rio2009a}. 

	Finally, Doob's maximal inequality 
	implies the second inequality of this proposition. Since Doob's 
	inequality is sharp, the constant is best possible, cf.\ 
	\cite[p.~87]{Bur1988a} and \cite[Theorem~2]{DG1978a}.
\end{proof}

\begin{proof}[Proof of Proposition~\ref{thm:Zakai}]
	The first inequality is proven by Zakai \cite[Theorem~1]{Zak1967a} 
	in case of $d=m=1$. We extend this proof to the case of a
	general $m$-dimensional Wiener processes and $\mathbb{R}^d$-valued integrands. 
	Due to Burkholder's inequality~\eqref{eq:cBurkholder1} and 
	the assumption on processes $f^{j}$, $j \in \{1,\dots,m\}$, 
	we have 
	$\big\| \sum_{j=1}^m \int_{0}^t f^{j}_u \, \mathrm{d} W^{j}_u
	\big\|_{L^{p}(\Omega; \mathbb{R}^d)} <\infty$.
	Let $\delta >0$ and $\varphi \in C^2(\mathbb{R}^d; \mathbb{R})$ 
	with $\varphi(x) = (\delta + \| x \|^2)^{\frac{p}{2}}$ for $x \in \mathbb{R}^d$.
	Applying It{\^o}'s formula for the function $\varphi$ to the It{\^o} process
	$(X_t)_{t \in [0,T]}$ with $X_t = \smash{ \sum_{j=1}^m \int_0^t f_u^j \, 
	\mathrm{d} W_u^j}$,
	then \cite[Equation~(6)]{Zak1967a} reads in the multi-dimensional case as
	\begin{align*}
		&\bigg( \delta+ \bigg\| \sum_{j=1}^m \int_{0}^t f^{j}_u \, \mathrm{d} W^{j}_u
		\bigg\|^2 \bigg)^{\frac{p}{2}} - \delta^{\frac{p}{2}} \\
		&\quad= \frac{p}{2} \int_{0}^t \bigg( \delta + \bigg\|
		\sum_{j=1}^m \int_{0}^s f^{j}_u \, \mathrm{d} W^{j}_u \bigg\|^2
		\bigg)^{\frac{p}{2}-1} \sum_{j=1}^m \sum_{i=1}^d (f^{i,j}_s)^2
		\, \mathrm{d} s \\
		&\qquad+ \frac{p(p-2)}{2} \int_{0}^t \bigg( \delta + \bigg\|
		\sum_{j=1}^m \int_{0}^s f^{j}_u \, \mathrm{d} W^{j}_u \bigg\|^2
		\bigg)^{\frac{p}{2}-2} \\
		&\qquad \qquad \times
		\sum_{l=1}^m \sum_{i,k=1}^d \bigg(\sum_{j=1}^m \int_{0}^s f^{i,j}_u \, 
		\mathrm{d} W^{j}_u \bigg) \bigg( \sum_{j=1}^m \int_{0}^s f^{k,j}_u \, 
		\mathrm{d} W^{j}_u \bigg) f^{i,l}_s f^{k,l}_s \, \mathrm{d} s \\
		&\qquad + p \sum_{l=1}^m \int_{0}^t \! \bigg( \delta + \bigg\|
		\sum_{j=1}^m \int_{0}^s f^{j}_u \, \mathrm{d} W^{j}_u \bigg\|^2
		\bigg)^{\!\frac{p}{2}-1} \sum_{i=1}^d \! \bigg( \sum_{j=1}^m \int_{0}^s \!\!
		f^{i,j}_u \, \mathrm{d} W^{j}_u \bigg) f^{i,l}_s \, \mathrm{d} W^l_s
	\end{align*}
	$\Prob$-a.s.
	Taking the expectation and using the Cauchy-Schwarz inequality, we obtain
	\begin{align*}
	&\Erw \bigg[ \bigg( \delta + \bigg\| \sum_{j=1}^m \int_{0}^t f^{j}_u \, 
	\mathrm{d} W^{j}_u \bigg\|^2 \bigg)^{\frac{p}{2}} \bigg] - \delta^{\frac{p}{2}} \\
	&\quad = \frac{p}{2} \int_{0}^t \Erw \bigg[ \bigg( \delta + \bigg\|
	\sum_{j=1}^m \int_{0}^s f^{j}_u \, \mathrm{d} W^{j}_u \bigg\|^2
	\bigg)^{\frac{p}{2}-1} \sum_{j=1}^m \| f^{j}_s \|^2 \bigg] \, \mathrm{d} s \\
	&\qquad + \frac{p(p-2)}{2} \int_{0}^t \Erw \bigg[ \bigg( \delta +
	\bigg\| \sum_{j=1}^m \int_{0}^s f^{j}_u \, \mathrm{d} W^{j}_u \bigg\|^2
	\bigg)^{\frac{p}{2}-2} \sum_{l=1}^m \sum_{i=1}^d \bigg(
	\sum_{j=1}^m \int_{0}^s f^{i,j}_u \, \mathrm{d} W^{j}_u \,f^{i,l}_s \bigg)^2
	\bigg] \, \mathrm{d} s \\
	&\quad \le \frac{p(p-1)}{2} \int_{0}^t \Erw \bigg[ \bigg( \delta + \bigg\|
	\sum_{j=1}^m \int_{0}^s f^{j}_u \, \mathrm{d} W^{j}_u \bigg\|^2
	\bigg)^{\frac{p}{2}-1} \sum_{j=1}^m \| f^{j}_s \|^2 \bigg] \, \mathrm{d} s ,
	\end{align*}
	which corresponds to the multi-dimensional variant of 
	\cite[inequality~(8)]{Zak1967a}. Then, inequality~\eqref{eq:Zakai} 
	follows from the same arguments as in \cite[pp.~171--172]{Zak1967a} 
	by applying a Gronwall-type inequality \cite[Lemma on p.~171]{Zak1967a} 
	and letting $\delta \to 0$. 
	
	Applying Doob's submartingale inequality
	finally yields the second inequality of the proposition.
\end{proof}
%
%
%
%
\bibliographystyle{plainurl}
\bibliography{Bibfile}
\end{document}